\documentclass[oneside,english,draft]{amsart}
\usepackage[T1]{fontenc}
\usepackage[latin9]{inputenc}
\usepackage{babel}
\usepackage{verbatim}
\usepackage{url}
\usepackage{amsthm}
\usepackage{amssymb}
\usepackage{booktabs}
\usepackage[ruled, vlined, linesnumberedhidden]{algorithm2e}
\usepackage[draft=false, unicode=true,
 bookmarks=false,bookmarksnumbered=false,bookmarksopen=false,
 breaklinks=false,pdfborder={0 0 1},backref=false,colorlinks=false]
 {hyperref}
\hypersetup{
 pdfauthor={Giosu{\`e} Muratore}}
 
\makeatletter
\numberwithin{equation}{section}
\numberwithin{figure}{section}
\numberwithin{table}{section}
\theoremstyle{plain}
\newtheorem{thm}{\protect\theoremname}[section]
\theoremstyle{definition}
\newtheorem{example}[thm]{\protect\examplename}
\theoremstyle{remark}
\newtheorem{rem}[thm]{\protect\remarkname}
\theoremstyle{definition}
\newtheorem{defn}[thm]{\protect\definitionname}
\theoremstyle{plain}
\newtheorem{cor}[thm]{\protect\corollaryname}
\theoremstyle{definition}
\newtheorem{notat}[thm]{\protect\notationname}
\theoremstyle{plain}
\newtheorem{lem}[thm]{\protect\lemmaname}

\usepackage{tikz}
\usetikzlibrary{positioning}

\makeatother

\providecommand{\definitionname}{Definition}
\providecommand{\examplename}{Example}
\providecommand{\remarkname}{Remark}
\providecommand{\theoremname}{Theorem}
\providecommand{\corollaryname}{Corollary}
\providecommand{\notationname}{Notation}
\providecommand{\lemmaname}{Lemma}

\begin{document}
\global\long\def\d{\mathrm{d}}%
\global\long\def\ev{\mathrm{ev}}%
\global\long\def\ctopT{c_{\mathrm{top}}^{T}}%
\global\long\def\ctop{c_{\mathrm{top}}}%
\global\long\def\ch{\mathrm{ch}}%
\global\long\def\td{\mathrm{td}}%

\title{Computations of Gromov--Witten invariants of toric varieties}
\author{Giosu{\`e} Muratore}
\begin{abstract}
We present the Julia package \verb!ToricAtiyahBott.jl!, providing an easy way to perform the Atiyah--Bott formula on the moduli space of genus $0$ stable maps $\overline{M}_{0,m}(X,\beta)$ where $X$ is any smooth projective toric variety, and $\beta$ is any effective $1$-cycle. The list of the supported cohomological cycles contains the most common ones, and it is extensible. We provide a detailed explanation of the algorithm together with many examples and applications.
The toric variety $X$, as well as the cohomology class $\beta$, must be defined using the package \verb!Oscar.jl!.
\end{abstract}

\address{CMAFcIO, Faculdade de Ci\^{e}ncias da ULisboa, Campo Grande 1749-016 Lisboa,
Portugal}
\email{\href{mailto:muratore.g.e@gmail.com}{muratore.g.e@gmail.com}
}
\urladdr{\url{https://sites.google.com/view/giosue-muratore}}
\maketitle

\section{Introduction}
The Gromov--Witten invariants of a variety $X$ are the degrees of certain cycles in the moduli space of stable maps $\overline{M}_{g,m}(X,\beta)$, where $\beta\in H_2(X,\mathbb Z)$ is an effective class. They are widely used to solve a large class of enumerative problems. For this reason, there is a great interest in finding explicit algorithms to compute Gromov--Witten invariants. Some algorithms are very efficient and can be applied to a wide range of cases (see, for example, the recent articles \cite{MR4575868,hu1,hu2}). 

In this article, we present an implementation of the Atiyah--Bott formula to compute \textit{any} Gromov--Witten invariant where $g=0$ and $X$ is a smooth projective toric variety. Our main references are \cite{MR1692603,MR3184181}.

In Section \ref{sec:2} we recall the main properties of toric varieties and Gromov--Witten invariants. In Section \ref{sec:3} we describe the algorithm together with the equivariant classes of the most common cohomological cycles. In the last section, we present some explicit computations and show some possible applications.

Our package is written in Julia (\cite{bezanson2017julia}) and is freely available at:
\begin{center} \url{https://github.com/mgemath/ToricAtiyahBott.jl} \end{center}
It requires \verb|Oscar.jl| (\cite{OSCAR,OSCAR-book}), in particular its toric geometry part (see \cite{bies}). We think that our package could be a useful tool to the mathematical community.

\noindent {\bf Acknowledgments.}
The author learned about toric varieties in the school BAGS in Nancy in 2018. He is truly grateful to the organizers, in particular to Gianluca Pacienza. 
Moreover, he thanks Csaba Schneider and Xiaowen Hu for mathematical support. Finally, he thanks Martin Bies, Janko Boehm, Lars Kastner, and Yue Ren for their help with \verb|Oscar.jl|. The author is a member of GNSAGA (INdAM). 
This project is supported by FCT - Funda\c{c}\~{a}o para a Ci\^{e}ncia e a Tecnologia,
under the project: UIDP/04561/2020  (\url{https://doi.org/10.54499/UIDP/04561/2020}).


\section{Gromov--Witten invariants on toric varieties}\label{sec:2}
In this section we briefly review toric varieties and Gromov--Witten invariants, in order to fix notation.
\subsection{Toric varieties}
A toric variety $X$ is an algebraic variety containing a torus $T=(\mathbb C^*)^n$ as a dense open subset, in such a way that the action of $T$ on itself extends to the whole $X$. The classification of toric varieties of dimension $n$ is equivalent to the classification of fans of cones in $\mathbb Z^n$. In particular, for each toric variety there is a fan $\Sigma$, uniquely defined modulo isomorphism. We refer to \cite{MR2810322} for a complete account on the correspondence between fans $\Sigma$ and toric varieties $X=X_\Sigma$.
In the following examples, we denote by $\{e_1,\ldots,e_n\}$ the standard generators of $\mathbb Z^n$.
\begin{example}\label{example:P}
    The projective space $\mathbb P^n$ is a toric variety with maximal cones $\{\sigma_0,\ldots,\sigma_n\}$ where $\sigma_i$ is the cone generated by the rays $\{e_0,\ldots,\hat{e}_i,\ldots,e_n\}$ and $e_0:=\sum_{i=1}^n-e_i$. 
\end{example}
Some natural constructions (direct product, blow-up along a toric invariant subvariety,...) are stable in the toric context. If $\mathcal{E}$ is a vector bundle of a smooth complete toric variety, then the projectivization $\mathbb P(\mathcal{E})$ is toric if and only if $\mathcal{E}$ is the direct sum of line bundles \cite[Proposition~A.2]{MR3491049}. The general strategy to find the fan of $\mathbb P(\mathcal{E})$ is described in \cite[Chapter~1 Section~7]{zbMATH03648907}.
\begin{example}\label{example:P(P)}
    Any Hirzebruch surface $\mathbb F_m=\mathbb P(\mathcal{O}_{\mathbb P^1}\oplus\mathcal{O}_{\mathbb P^1}(m))$ is toric. The maximal cones of its fan are $$\{\{e_1,e_2\},\{e_1,-e_2\},\{e_2,-e_1-me_2\},\{-e_2,-e_1-me_2\}\}.$$
    Analogously, $\mathbb P(\mathcal{O}_{\mathbb P^2}\oplus\mathcal{O}_{\mathbb P^2}(m))$ is toric. The maximal cones of its fan are 
    $$\bigcup_{\mathfrak e=\pm e_3}\{
    \{e_1,e_2,\mathfrak e\},
    \{e_2,\mathfrak e,-e_1-e_2-me_3\},
    \{e_1,\mathfrak e,-e_1-e_2-me_3\}\}.$$
\end{example}

\begin{notat}\label{notat:toric}
    From now on, for any toric variety $X$ relative to a fan $\Sigma$, we use the following notation.
    \begin{itemize}
        \item $\Sigma^{(k)}$ is the subset of $k$-dimensional cones of $\Sigma$. So, $\Sigma^{(n)}$ are the maximal cones and they are usually denoted by $\sigma$.
        \item For each $\sigma\in\Sigma^{(n)}$, we denote by $\sigma^*$ the set of all $\sigma'\in\Sigma^{(n)}$ such that $\sigma\cap\sigma'\in\Sigma^{(n-1)}$.
        \item $\Sigma^{(1)}=\{\rho_1,\ldots,\rho_r\}$ are the $1$-dimensional cones (rays) of $\Sigma$.
        \item For each cone $\tau\in\Sigma^{(k)}$, $V(\tau)$ is the $T$-invariant subvariety of dimension $(n-k)$ defined by $\tau$.
    \end{itemize}
\end{notat}
\subsection{Moduli space of maps}
Let us briefly recall the main properties of the moduli space of stable maps. The toric variety $X$ is smooth and projective.
\begin{itemize}
\item $\overline{M}_{0,m}(X,\beta)$ is the coarse moduli space of genus
$0$ stable maps to $X$ of class $\beta$ with $m$ marked points.
\item $\ev_{j}\colon \overline{M}_{0,m}(X,\beta)\rightarrow X$
is the evaluation map at the $j^{\mathrm{th}}$ marked point.
\item $\delta\colon \overline{M}_{0,m+1}(X,\beta)\rightarrow\overline{M}_{0,m}(X,\beta)$
is the forgetful map of the last point.
\end{itemize}

Usually, the moduli points are denoted by $(C,f,p_1,\ldots,p_m)$ where $f\colon C\rightarrow X$ is the morphism and $p_1,\ldots,p_m\in C$ are the marked points. We denote by $\mathbb L_i$ the line bundle whose fiber at $(C,f,p_1,\ldots,p_m)$ is $\Omega_{C,p_i}$. Finally, $\psi_i$ is the first Chern class of $\mathbb L_i$. The virtual dimension of $\overline{M}_{0,m}(X,\beta)$ is
\begin{equation}\label{eq:virdim}
\mathrm{virdim}(\overline{M}_{0,m}(X,\beta)) = \dim(X) - \deg(c_1(\Omega_X)\cdot \beta) +m-3.
\end{equation}
Given a vector bundle $\mathcal{E}$ or rank $r$, cohomology classes $\gamma_1,\ldots,\gamma_m\in H^*(X,\mathbb C)$, and non negative integers $a_1,\ldots,a_m$, we define a Gromov--Witten invariant to be the following degree of a cohomology class of maximal codimension:
\begin{equation}\label{eq:GW}
    \int_{\overline{M}_{0,m}(X,\beta)}\ev_{1}^{*}(\gamma_{1})\cdots\ev_{m}^{*}(\gamma_{m})\cdot \psi_1^{a_1}\cdots\psi_m^{a_m}\cdot c_r(\mathcal{E}).
\end{equation}
The action of $T$ on $X$ lifts to $\overline{M}_{0,m}(X,\beta)$ by composition. The fixed point locus of this new action is the disjoint union of subvarieties of $\overline{M}_{0,m}(X,\beta)$ parameterized by all possible decorated graphs of $\overline{M}_{0,m}(X,\beta)$.

\begin{defn}
    A decorated graph of $\overline{M}_{0,m}(X,\beta)$ is the isomorphism class of the tuple $\Gamma=(g, \mathbf{c}, \mathbf{w}, \nu)$, such that
    \begin{enumerate}
\item $g$ is a tree; that is $g$ is a simple undirected, connected graph without cycles. We denote by $V_g$ and $E_g$ the sets of vertices and edges of $g$.
\item The vertices of $g$ are colored by the maximal cones of $\Sigma$; the coloring is given by a map $\mathbf c\colon V_g \rightarrow \Sigma^{(n)}$ such that: if $v$ and $v'$ are two vertices in the same edge, then ${\mathbf{c}(v)}\cap{\mathbf{c}(v')}\in\Sigma^{(n-1)}$.
\item Each edge $e=(v,v')$ is weighted by a map $\mathbf w\colon E_g\rightarrow \mathbb Z$, such that $\mathbf w(e)>0$ and $\sum_{e\in E_g} \mathbf w(e)\cdot [V({\mathbf{c}(v)}\cap{\mathbf{c}(v')})]=\beta$.
\item The marks of the vertices of $g$ are given by a map $\nu\colon A\rightarrow V_{g}$, where $A=\{1,\ldots, m\}$ if $m>0$, and $A=\emptyset$ if $m=0$.
    \end{enumerate}
\end{defn}
An automorphism of $\Gamma$ is an automorphism of $g$ compatibles with $\mathbf{c},\mathbf{w},\nu$.
    We denote by $a_\Gamma$ the integer
    $$a_\Gamma:=|\mathrm{Aut}(\Gamma)|\prod_{e\in E_g}\mathbf w(e),$$
    and by $F_\Gamma$ the fixed point locus relative to $\Gamma$.

Following \cite{AB}, given a cohomology class $\alpha\in H^*(\overline{M}_{0,m}(X,\beta),\mathbb C)$, we denote by 
 $\alpha^T(\Gamma)$ 
the $T$-equivariant cohomology class induced by $\alpha$. Finally, we have the integration formula.
\begin{thm} Let $P$ be a symmetric polynomial in Chern classes of equivariant vector bundle of $\overline{M}_{0,m}(X,\beta)$. Then 
    \begin{equation}\label{eq:ABt}
\int_{\overline{M}_{0,m}(X,\beta)}P=\sum_{\Gamma}\frac{1}{a_{\Gamma}}\frac{P^{T}(\Gamma)}{\ctopT(N_{\Gamma})(\Gamma)},
\end{equation}
where $N_\Gamma$ is the normal tangent bundle of $F_\Gamma$.
\end{thm}
See \cite{kontsevich1995enumeration,graber1999localization, MR1692603, MR3184181}.
Finally, we recall the definition of the small quantum cup product of $X$. 
Let $\{\theta_1,\ldots, \theta_t\}$ be a basis of the vector space $H^*(X,\mathbb Q)$ and let $\{\vartheta_1,\ldots, \vartheta_t\}$ be its dual basis with respect to the cup product. Given any two classes $\alpha_1,\alpha_2\in H^*(X,\mathbb Z)$ we define the product
\begin{equation}\label{eq:star_prod}
    \alpha_1 \star \alpha_2 = \alpha_1\cdot\alpha_2+
    \sum_{\beta}q^\beta\sum_{i=1}^t \vartheta_i\int_{\overline{M}_{0,3}(X,\beta)}
\ev^*_1(\alpha_1)\cdot\ev^*_2(\alpha_2)\cdot\ev^*_3(\theta_i).
\end{equation}
We extend it by linearity to the whole small quantum ring $QH^s(X)$.

\section{The algorithm}\label{sec:3}
In this section we describe the main features of our implementation of Equation~\eqref{eq:ABt}. The core part of the algorithm consists of four nested \verb!for! loops that generate the decorated graphs  $\Gamma=(g, \mathbf{c}, \mathbf{w}, \nu)$. We used the algorithm of \cite{McKay} in order to generate all trees. So that, the first \verb!for! loop runs among all trees with at most $p$ edges. The number $p$ corresponds to the maximal number of $T$-invariant curves $\{C_i\}_{i=1}^p$ such that
$$\beta = [C_1]+[C_2]+\cdots+[C_p].$$
Using the fact that the Mori cone is generated by $T$-invariant curves and it is dual to the cone of nef divisors \cite[\S6.3]{MR2810322}, we have that $p= \beta\cdot (M_1+\cdots+M_k)$ where $M_i$ are the classes of the generators of the nef cone.

The generation of $\mathbf{c}$ and $\nu$ modulo isomorphisms is made using standard properties of trees. In order to generate the functions $\mathbf{w}\colon E_g\rightarrow \mathbb Z$, we proceed in the following way. 
Let us fix a tree $g$ and a coloring $\mathbf{c}\colon V_g\rightarrow \Sigma^{(n)}$.
By definition, for any map $\mathbf{w}\colon E_g\rightarrow \mathbb Z$ we have
\begin{equation}\label{eq:b}
    \beta = \sum_{e\in E_g}\mathbf{w}(e)\cdot [V(\mathbf{c}(v)\cap \mathbf{c}(v'))].
\end{equation}
Thus, if $\bar e\in E_g$ is a fixed edge, since $\mathbf{w}(e)>0$ we have that
\begin{equation}\label{eq:w}
    \beta - \mathbf{w}(\bar e)\cdot [V(\mathbf{c}(v)\cap \mathbf{c}(v'))]-\sum_{e\in E_g\setminus\{\bar e\}} [V(\mathbf{c}(v)\cap \mathbf{c}(v'))]
\end{equation}
is an effective $1$-cycle. Thus, an upper bound for $\mathbf{w}(\bar e)$ is given by the number $b_{\bar e}$ that is the maximal number that, in place of $\mathbf{w}(\bar e)$, makes \eqref{eq:w} an effective cycle. The package \verb|Oscar.jl| provides all divisors that span the extremal rays of the nef cone\footnote{I thank Martin Bies for this feature.}, so \eqref{eq:w} is effective if and only if it intersects non negatively those nef divisors \cite[Theorem~6.3.20]{MR2810322}. Hence we can find $b_{\bar e}$. Now, we generate all possible functions $\mathbf{w}$ such that $1\le \mathbf{w}(e)\le b_{e}$ for all $e\in E_g$ and verifying \eqref{eq:b}.

In this way, we generate all decorated graphs. The equivariant cycles are implemented as Julia functions that take a graph as input and return a rational number. In order to implement \eqref{eq:ABt}, we sum all these rational numbers to a variable \verb|ans|. Eventually, 
the value of \verb|ans| will be the value of the desired Gromow--Witten invariant.

  
For completeness, we list the explicit formulas of the equivariant classes we used, whose proofs can be found in \cite{kontsevich1995enumeration,cox1999mirror,MR1692603,MR1928909,MR3184181,MR4634985,MR4383164}.
Some proofs are worked out only for $\mathbb P^n$, but they extend tautologically for any smooth projective toric variety. We try to be as clear and simple as possible. We use the notation introduced in \ref{notat:toric}. The following is Lemma 6.10 in Spielberg's thesis. 
\begin{lem}
    Let $\sigma_1,\sigma_2\in\Sigma^{(n)}$ be such that $\tau=\sigma_1\cap\sigma_2\in\Sigma^{(n-1)}$. 
    
    Let $\{\rho_{i_1},\ldots,\rho_{i_{n-1}}\}$ be the rays of $\tau$. Moreover, let $\rho'$ be the ray of $\sigma_1$ that is not in $\tau$. 
    Fix $\omega_1,\ldots,\omega_r$ weights of the action of $(\mathbb C^*)^r$ on 
    $\mathbb C^r$. 
    The induced $\mathbb C^*$-action on the subvariety
    $V(\tau)$ has weight 
    $\omega^{\sigma_1}_{\sigma_2}$ at the point $V(\sigma_1)$ given by:
    $$\omega^{\sigma_1}_{\sigma_2}:=\sum_{i=1}^r\langle\rho_i,u_n\rangle\omega_i$$
where $\{u_1,\ldots,u_n\}$ is the $\mathbb Z$-dual basis of $\{\rho_{i_1},\ldots,\rho_{i_{n-1}},\rho'\}\subset\mathbb Z^r$.
\end{lem}

Suppose that we fix a decorated graph $\Gamma$. A flag of a vertex $v$ is the pair $F=(v,e)$ where $e$ is an edge of $g$ containing $v$. We denote by $F_v$ the set of flags of $v$. Let $d=\mathbf{w}(e)$, $\sigma_1$ be the color of $v$, and $\sigma_2$ be the color of the other vertex of $e$. We denote by 
$$\omega_F:=\frac{1}{d}\omega^{\sigma_1}_{\sigma_2}.$$
Given $\gamma\in\sigma_1^*$ we denote by $\lambda^\gamma_e$ the degree of the class $d[V(\sigma_1\cap\sigma_2)]\cdot [V(\rho)]$, where $\rho$ is the ray of $\sigma_1$ that is not a ray of $\gamma$. We define by
\begin{eqnarray*}
    \Delta(e)&:=& \frac{(-1)^dd^{2d}}{(d!)^2(\omega^{\sigma_1}_{\sigma_2})^{2d}}\prod_{\gamma\in\sigma_1^*\setminus\{\sigma_2\}}\prod_{i \in K}\left(\omega^{\sigma_1}_\gamma-\frac{i}{d}\omega^{\sigma_1}_{\sigma_2} \right)^{-\mathrm{sgn}(\lambda^\gamma_e+1)} \\
    \Xi(\Gamma) &:=& \prod_{v\in V_g}\left(\prod_{\gamma\in\mathbf{c}(v)^*}\omega^{\mathbf{c}(v)}_\gamma \right)^{|F_v|-1}\prod_{e\in E_g}\Delta(e),
\end{eqnarray*}
where $K=\{0,\ldots, \lambda^\gamma_e\}$ if $\lambda^\gamma_e\ge0$, $K=\{\lambda^\gamma_e+1,\ldots,-1\}$ if $\lambda^\gamma_e\le-2$, and $K=\emptyset$ otherwise. Let $\{a_i\}_{i=1}^m$ be non negative integers.
Let $S_v:=\nu^{-1}(v)=\{j_1,...,j_k\}$ denote the set of marks mapped to $v$, and let $\overline{S_v}:=\sum_{i\in S_v}a_i$ and $n(v):=|F_v|+|S_v|$. If $m=0$ or $S_v=\emptyset$, we set $\overline{S_v}:=0$. If $n(v)-3\ge\overline{S_v}$, we recall the multinomial coefficients
\begin{equation*}
    \binom{n(v)-3}{n(v)-3-\overline{S_v},a_{j_1},...,a_{j_k}} = \frac{(n(v)-3)!}{(n(v)-3-\overline{S_v})!\cdot a_{j_1}!\cdots a_{j_k}!}.
\end{equation*}
We may extend the multinomial in the following way:
\begin{itemize}
    \item if $\overline{S_v}>n(v)-3\ge0$, then its value is $0$;
    \item if $\overline{S_v}=0$, then its value is $1$;
    \item if $S_v=\{a_{j_1}\}$ and $n(v)=2$, then its value is $(-1)^{a_{j_1}}$.
\end{itemize}
Let us denote by $\Psi$ the following class
\begin{eqnarray}\label{eq:Psi}
\Psi(\Gamma) & :=& \prod_{v\in V_g} \binom{n(v)-3}{n(v)-3-\overline{S_v},a_{j_1},...,a_{j_k} }
\left(\sum_{F\in F_v}\omega_F^{-1} \right)^{-\overline{S_v}}.
\end{eqnarray}

Finally, the equivariant class of the inverse of $\ctopT(N_\Gamma)$ times the $\psi$-classes is
\begin{eqnarray}\label{eq:ctopKont}
\frac{[\psi_1^{a_1}\cdots \psi_m^{a_m}]^T(\Gamma)}{\ctopT(N_{\Gamma})(\Gamma)} & = & \Psi(\Gamma) \Xi(\Gamma)\prod_{v\in V_g}\prod_{F\in F_v}\omega_F^{-1}\left(\sum_{F\in F_v}\omega_F^{-1} \right)^{n(v)-3}.
\end{eqnarray}

In the package, we split \eqref{eq:ctopKont} in two different functions called \verb!Psi! and \verb|Euler_inv|. The first one implements \eqref{eq:Psi}, the latter implements \eqref{eq:ctopKont} when all $a_i$ are trivial.

So, we clarify the contribution of $\psi$ classes and $\ctop(N_\Gamma)$ to \eqref{eq:ABt}. Let us see the other classes. 

For any maximal cone $\sigma\in\Sigma^{(n)}$, for any ray $\rho\in\Sigma^{(1)}$, and for any non negative integer $k$, we define the following 
\begin{equation*}
    \lambda(\sigma, \rho, k) := \begin{cases}
  0  & \text{if $\rho$ is not a ray of $\sigma$, and $k>0$}, \\
  1 & \text{if $\rho$ is not a ray of $\sigma$, and $k=0$}, \\
  (\omega^{\sigma}_\gamma)^k & \text{if $\gamma\in\sigma^*$, and $\rho$ is a ray of $\sigma$ but not of $\gamma$.}
\end{cases}
\end{equation*}
The cone $\gamma$ with the property of the last case exists and it is unique, so the definition makes sense.
Let $Z$ be the following cohomology class:
\begin{equation}\label{eq:Z}
    Z=[V(\rho_1)]^{k_1}\cdot [V(\rho_2)]^{k_2}\cdots [V(\rho_r)]^{k_r}.
\end{equation}
For any maximal cone $\sigma$, let us define
\begin{equation}\label{eq:lambda}
    \Lambda(\sigma, Z):= \prod_{j=1}^r \lambda(\sigma,\rho_j,k_j).
\end{equation}
Since the cohomology ring of $X$ is generated by invariant divisors \cite[Chapter~12]{MR2810322}, we may extend \eqref{eq:lambda} by linearity to any homogeneous class $Z\in H^*(X,\mathbb Q)$.

For any $i=1,\ldots,m$ the contribution of $\ev_i^*(Z)$ is
\begin{equation}\label{eq:ev}
\ev_i^*(Z)^T(\Gamma)=\Lambda(\mathbf{c}(\nu(i)),Z).
\end{equation}
Let $M$ be a line bundle of $X$. We extend \eqref{eq:lambda} and \eqref{eq:ev} to $M$ by taking $Z=c_1(M)$.
Finally, let us assume that $h^1(X,M)=0$. 
We define the integer
$$M_e:=d(c_1(M)\cdot[V(\mathbf{c}(v)\cap \mathbf{c}(v'))]).$$
The equivariant class of the Euler class of the bundle $\delta_*(\ev_{m+1}^*(M))$ is
\begin{multline*}
    \ctopT(\delta_*(\ev_{m+1}^*(M)))(\Gamma)=\prod_{v\in V_g, |F_v|\neq 1}\Lambda(\mathbf{c}(v),M)^{1-|F_v|} \\
    \times\left(\prod_{e\in E_g, M_e>0}
\prod_{\alpha=0}^{M_e}\frac{\alpha \Lambda(\mathbf{c}(v),M)-(M_e-\alpha)\Lambda(\mathbf{c}(v'),M)}{M_e}\right).
\end{multline*}


\begin{rem}
    The algorithm we described, as well as the examples in the next section, are relative to Version 1.0.0 of the package and Version 0.13.0 of \verb|Oscar.jl|. Some part of the code may be improved in the future. We refer to the internal documentation of the package for an up-to-date explanation of its features.
\end{rem}
\begin{rem}
    Let us consider the toric variety $X = \mathbb{P}^n$. Every cohomology class is an integer multiple of a unique class (namely, the class of a linear subspace). This implies that the generation of all maps $\mathbf{w}\colon E_g \rightarrow \mathbb{Z}$ is equivalent to finding the partition of an integer. Moreover, since any two maximal cones share a facet, the condition $\mathbf{c}(v) \cap \mathbf{c}(v') \in \Sigma^{(n-1)}$ on a coloring $\mathbf{c}\colon V_g \rightarrow \Sigma^{(n)}$ can be replaced with the simpler condition $\mathbf{c}(v) \neq \mathbf{c}(v')$. Finally, one can see that $\omega_{\sigma_2}^{\sigma_1} = \omega_{i_1} - \omega_{i_2}$, where the ray $\rho_{i_1}$ (resp., $\rho_{i_2}$) is the unique ray not in $\sigma_1$ (resp., $\sigma_2$).

    Csaba Schneider and the author created a package, similar to the present one, for computations of Gromov--Witten invariants in projective spaces, see \cite{MR4383164}. That package is faster because it uses the simplifications we explained. In other words, it does not deal with tedious computations with the polyhedral fan and the cohomology ring of $X$, but only with classical combinatorial functions on integers. It is also easier to install as it does not require \verb|Oscar.jl|. Of course, all computations of that package are supported by this one.
\end{rem}
\begin{rem}
    In the published version of this paper \cite{MR4739375}, $\Xi(\Gamma)$ is erroneously reported without the exponent $|F_v|-1$.
\end{rem}

\section{Explicit computations}
In this section we use our package to perform some explicit computations, in order to show its functionalities. For more examples, please consult the internal documentation of the package. 
The typical use of the package is the following. First, one defines the variety $X$ and the class $\beta$ using \verb!Oscar.jl!, then defines the function \verb!P! by combining the functions that are built-in functions of the package. Finally, one calls \verb!IntegrateAB(X, beta, m, P)! to get the result of the integration.
We provided our package with the following functions.
\begin{itemize}
    \item \verb!a_point(X)! returns the cohomology class $[\mathrm{pt}]$ of a point of $X$.
    \item \verb|moment_graph(X)| returns the moment graph of $X$.
\end{itemize}
For example, if \verb|mg=moment_graph(X)|, then \verb|mg[1,2]| is the cohomology class of the $T$-invariant curve $V(\sigma)\cap V(\sigma')$, where $\sigma,\sigma'$ are the first two elements of $\Sigma^{(n)}$.
\subsection{Enumeration of curves}
Let $h=c_1(\mathcal{O}_{\mathbb P^3}(1))$. The following computation
\begin{equation*}
    \int_{\overline{M}_{0,5}(\mathbb{P}^{3},2)}
\ev_1^*(h^3)\cdot\ev_2^*(h^3)\cdot\ev_3^*(h^3)\cdot\ev_4^*(h^2)\cdot\ev_5^*(h^2)
\end{equation*}
equals the number of conics in $\mathbb{P}^3$ passing through three points and two lines in general position (see \cite[Lemma~14]{FP}). It can be computed in the following way
\begin{verbatim}
P3 = projective_space(NormalToricVariety, 3); # the space P^3.
h = cohomology_class(toric_line_bundle(P3, [1]));
P = ev(1, h^3)*ev(2, h^3)*ev(3, h^3)*ev(4, h^2)*ev(5, h^2);
IntegrateAB(P3, 2*h^2, 5, P);
Result: 1
\end{verbatim}
Let $d\ge0$ and $e\le0$ such that $2d+e>0$. We want to compute the numbers of degree $d$ rational curves in $\mathbb P^3$ meeting $2d+e$ generic points, and in addition passing through a fixed point with global multiplicity $-e$. By \cite[Example~8.2]{MR1832328} we know that such a number is
\begin{equation*}
    \int_{\overline{M}_{0,2d+e}(X,\beta)}\prod_{i=1}^{2d+e}
\ev_i^*([\mathrm{pt}]),
\end{equation*}
where $\pi\colon X\rightarrow\mathbb{P}^3$ is the blow-up at one point, $H=\pi^*(h^2)$, $E$ is the class
of a line in the exceptional divisor, and finally $\beta=dH+eE$. This invariant can be computed using the package in the following way.
\begin{verbatim}
X = domain(blow_up(P3, ideal(gens(cox_ring(P3))[1:3])));
mg = moment_graph(X);
H = mg[4,5];
E = mg[1,2];
d = 1; # this can be any non negative integer
e = -1; # this can be any non positive integer
beta = d*H+e*E;
P = prod(i -> ev(i, a_point(X)), 1:(2*d+e)); # we require 2d+e>0
IntegrateAB(X, beta, 2*d+e, P);
Result: 1
\end{verbatim}
In the Gromov--Witten invariant \eqref{eq:GW}, the cohomology class $\psi_1^{a_1}\cdots\psi_m^{a_m}$ is implemented using the command \verb|Psi(a_1,...,a_m)|. We may use this to compute the number of curves with tangency conditions. For example, the number of lines in $\mathbb P^3$ passing through a general point of a quartic surface with multiplicity $3$ is
\begin{equation*}
    \int_{\overline{M}_{0,1}(\mathbb{P}^{3},1)}
\frac{1}{4}\left (\ev_1^*(4h)^3+ 3\ev_1^*(4h)^2\psi_1+ 2\ev_1^*(4h)\psi^2_1\right )\ev_1^*(h)^2=2,
\end{equation*}
see \cite{MR3773793,MR4332489,MR4256011}. We compute it using the following code.
\begin{verbatim}
l = 4*h;
P = 1//4*(ev(1,l)^3+3*ev(1,l)^2*Psi(1)+2*ev(1,l)*Psi(2))*ev(1,h)^2;
IntegrateAB(P3, h^2, 1, P);
Result: 2
\end{verbatim}

\subsection{Quantum cohomology ring of a threefold} In \cite[Section~9.3]{MR1692603}, Spielberg computed the small quantum cohomology ring of 
$X=\mathbb{P}(\mathcal{O}_{\mathbb{P}^2}\oplus\mathcal{O}_{\mathbb{P}^2}(2))$. His strategy was the following. Among all Gromov--Witten invariants that appear in the definition of quantum product, only a finite number are non zero. Using some clever arguments he can list all of them. After that, he computes each one by hand, using Equation~\eqref{eq:ABt}. In this subsection we will use our code in order to get the computations he made by hand. In order to define the toric variety $X$ using \verb!Oscar.jl!, first we list its rays and maximal cones, and then we use the command \verb!normal_toric_variety!. We use Spielberg's notation instead of that of Example~\ref{example:P(P)}.
\begin{verbatim}
e1 = [1,0,0]; e2 = [0,1,0]; e3 = [0,0,1];
ray_gens = [e1, -e1, e2, e3, -e2-e3-2*e1];
max_cones = [[1,3,4],[1,3,5],[1,4,5],[2,3,4],[2,3,5],[2,4,5]];
X=normal_toric_variety(ray_gens,max_cones;non_redundant=true);
\end{verbatim}
Let us define an array containing the generators of $H^*(X,\mathbb C)$,
\begin{verbatim}
Z = [cohomology_class(X, Z) for Z in gens(cohomology_ring(X))];
\end{verbatim}
With this notation, the cohomology ring of $X$ is
\begin{equation*}
    H^*(X,\mathbb C) = \frac{\mathbb C[Z_1,Z_2,Z_3,Z_4,Z_5]}{(Z_1-Z_2-2Z_5, Z_3-Z_4, Z_3-Z_5,Z_1Z_2,Z_3Z_4Z_5)}
\end{equation*}
The generators of the Mori cone, that in the paper are denoted by $\lambda_1$ and $\lambda_2$, are
\begin{verbatim}
mg = moment_graph(X);
lambda1 = mg[1,4];
lambda2 = mg[4,5];
\end{verbatim}
The first non zero invariant computed by Spielberg is
\begin{equation*}
    \int_{\overline{M}_{0,3}(X,\lambda_2)}
\ev_1^*(Z_3)\cdot\ev_2^*(Z_3)\cdot\ev_3^*(Z_4Z_5)=-1.
\end{equation*}
In his paper it is denoted by $\Phi^{\lambda_2}(Z_3,Z_3,Z_4Z_5)$. In our package this is computed using the following lines:
\begin{verbatim}
P = ev(1, Z[3])*ev(2, Z[3])*ev(3, Z[4]*Z[5]);
IntegrateAB(X, lambda2, 3, P);
\end{verbatim}
The result is as expected. The next non zero invariant is
\begin{equation*}
    \Phi^{\lambda_1}(Z_1,Z_1,Z_2Z_3Z_4)=\int_{\overline{M}_{0,3}(X,\lambda_2)}
\ev_1^*(Z_1)\cdot\ev_2^*(Z_1)\cdot\ev_3^*(Z_2Z_3Z_4)=1.
\end{equation*}
Again, the same computation with our code gives the same result:
\begin{verbatim}
P = ev(1, Z[1])*ev(2, Z[1])*ev(3, Z[2]*Z[3]*Z[4]);
IntegrateAB(X, lambda1, 3, P);
Result: 1
\end{verbatim}
We tested all computations in Spielberg's tesis. Those computations, as well as all computations in this article, are publicly available for reader's convenience\footnote{The repository is \url{https://github.com/mgemath/ToricAB_computations}.}.

\subsection{Quantum cohomology ring of a fourfold}
Spielberg's strategy works for any toric Fano variety even if the number of necessary computations grows. Take the toric Fano fourfold $X=\mathbb{P}(\mathcal{O}_{\mathbb{P}^3}\oplus\mathcal{O}_{\mathbb{P}^3}(1))$, which appears in Batyrev's classification \cite{MR1703904,MR1772804}. We implements $X$ in the following way.
\begin{verbatim}
X = proj(toric_line_bundle(P3, [0]), toric_line_bundle(P3, [1]));
\end{verbatim}
Take two classes $\alpha_1,\alpha_2\in H^2(X,\mathbb Z)$ and the class of a curve $\beta$. By Equation~\eqref{eq:virdim}, a necessary condition for the invariant
\begin{equation}\label{eq:invariant}
    \int_{\overline{M}_{0,3}(X,\beta)}
\ev_1^*(\alpha_1)\cdot\ev_2^*(\alpha_2)\cdot\ev_3^*(\theta)
\end{equation}
to be non zero is that the codimension of the class $\theta$ is equal to $2-\deg(c_1(\Omega_X)\cdot\beta)$. Thus the codimension of $\theta$ is $3$ or $4$. If it is $3$, then $-\deg(c_1(\Omega_X)\cdot\beta)=1$. But this is impossible since the Mori cone of $X$ is generated by the classes $\lambda_1$ and $\lambda_2$ where
\begin{verbatim}
mg = moment_graph(X);
lambda1 = mg[1,2];
lambda2 = mg[4,8];
\end{verbatim}
and a direct computation shows $-\deg(c_1(\Omega_X)\cdot\lambda_1)=2$, and $-\deg(c_1(\Omega_X)\cdot\lambda_2)=3$.\newline
So that, the invariant \eqref{eq:invariant} is non zero necessary if $\theta$ has codimension $4$ (in particular, it is the class of a point) and $\beta = \lambda_1$. The usual cohomology ring of $X$ is
\begin{equation*}
    H^*(X,\mathbb C) = \frac{\mathbb C[Z_1,Z_2,Z_3,Z_4,Z_5,Z_6]}{(Z_1 - Z_6, Z_2 - Z_6, Z_3 - Z_6, Z_1 + Z_4 - Z_5, Z_4Z_5, Z_1Z_2Z_3Z_6)}
\end{equation*}
As before, we implement it using the following command.
\begin{verbatim}
Z = [cohomology_class(X, Z) for Z in gens(cohomology_ring(X))];
\end{verbatim}
We want to compute $Z_4\star Z_5$. By Equation~\eqref{eq:star_prod}, we have
$$
Z_4\star Z_5 = Z_4\cdot Z_5+q^{\lambda_1} \int_{\overline{M}_{0,3}(X,\lambda_1)}
\ev_1^*(Z_4)\cdot\ev_2^*(Z_5)\cdot\ev_3^*([\mathrm{pt}]).
$$
The code to compute it is the following:
\begin{verbatim}
P = ev(1, Z[4])*ev(2, Z[5])*ev(3, a_point(X));
IntegrateAB(X, lambda1, 3, P);
Result: 1
\end{verbatim}
Thus $Z_4\star Z_5 = q^{\lambda_1}$. Following this path one can recover the entire multiplication table of the ring $QH^s(X)$.

\subsection{Twisted Gromov--Witten invariants}
Let $M$ be a line bundle on $X$ with $h^1(X,M)=0$. The integral
\begin{equation*}
    \int_{\overline{M}_{0,m}(X,\beta)}\ev_{1}^{*}(\gamma_{1})\cdots\ev_{m}^{*}(\gamma_{m})\cdot \psi_1^{a_1}\cdots\psi_m^{a_m}\cdot \ctop(\delta_*(\ev^*_{m+1}(M)))
\end{equation*}
gives Gromov--Witten 
invariants
of a subvariety of $X$  
cut out by a section of $M$ (see, e.g., \cite[Example~B]{MR2510741}). In the package, in order to insert $\ctop(\delta_*(\ev^*_{m+1}(M)))$ we use the command \verb!push_ev(M)!.

For example, take $X=\mathbb{P}(\mathcal{O}_{\mathbb{P}^3}\oplus\mathcal{O}_{\mathbb{P}^3}(1))$ that we used in the last subsection, and $M=\det(\Omega_X)^\vee$ the anticanonical bundle. A section of $M$ cuts out a Calabi--Yau threefold and its Gromov--Witten invariants
\begin{align*}
\int_{\overline{M}_{0,1}(X,\lambda_1)}\psi_1^1c_3(\delta_*(\ev^*_{2}(M)))&=-120,        &  \int_{\overline{M}_{0,0}(X,\lambda_2)}c_4(\delta_*(\ev^*_{1}(M)))&=27,
\end{align*}
can be computed in the following way.
\begin{verbatim}
M = anticanonical_bundle(X);
P = push_ev(M);
IntegrateAB(X, lambda1, 1, Psi(1)*P);
Result: -120
IntegrateAB(X, lambda2, 0, P);
Result: 27
\end{verbatim}

We refer to the page on Github for a list of all 
implemented classes.

\bibliographystyle{amsalpha}
\bibliography{computations}

\end{document}